\documentclass[10pt,reqno]{amsart}
\usepackage{amssymb}
\usepackage{amsmath, amssymb, amsthm, color}
\usepackage{extarrows}
\usepackage{paralist}
  \usepackage{graphics, graphicx} 
  \usepackage{epsfig} 
  \usepackage{psfrag,epstopdf}
\usepackage{subeqnarray}
\usepackage{cases}
\usepackage{cite}
\usepackage{bibentry}
\usepackage{setspace}

\newtheorem{theorem}{Theorem}[section]
\newtheorem{lemma}[theorem]{Lemma}

\newtheorem*{thma}{Theorem A}

\theoremstyle{definition}

\newtheorem{remark}[theorem]{Remark}

\newcommand{\rone}{\mathbb{R}}

\newcommand{\calY}{{\mathcal Y}}
\newcommand{\calZ}{{\mathcal Z}}

\newcommand{\tilv}{{\tilde v}}

\newcommand{\f}{\frac}

\newcommand{\pa}{\partial}

\newcommand{\qand}{{\quad \text{and} \quad}}

\newcommand{\qin}{{\quad \text{in} \quad}}
\newcommand{\qon}{{\quad \text{on} \quad}}

\newcommand{\vae}{{\varepsilon}}
\newcommand{\eps}{\varepsilon}

\newcommand{\la}{{\lambda}}

\newcommand{\leps}{\lambda_\varepsilon}
\newcommand{\psiuv}{\psi(u_0+v)}

\begin{document}

\title[Generalized Abelian Chern-Simons-Higgs model on a torus]
{Existence of multiple   solutions for the generalized Abelian Chern-Simons-Higgs model on a torus}
 \author{Jongmin Han and Kyungwoo Song$^\dagger$  }
 \date{}

\address{Department of Mathematics, College of Sciences, Kyung Hee University, Seoul 02447,  Korea}

\email{jmhan@khu.ac.kr, kyusong@khu.ac.kr}



\maketitle
\begin{abstract}
We construct multiple solutions of the generalized self-dual  Abelian Chern-Simons-Higgs equation 
in a two dimensional flat torus by the topological degree method. 
\end{abstract}


\section{Introduction}

Let $\Omega$ be a two dimensional flat torus.
In this paper, we are concerned with the generalized self-dual Abelian Chern-Simons-Higgs equation on $\Omega$:
\begin{align}
\label{eq:gcsh}
(1-e^U) \Delta U -  e^U |\nabla U|^2  =  - \frac{1}{\eps^2} e^U (e^U-1)^2 + 4\pi \sum^{d}_{j=1}n_j \delta_{p_j} 
\end{align}
Here, $\Gamma=\{p_1,\cdots,p_d\} \subset \Omega$, $n_j$ is a positive integer, and $\delta_p$ is the Dirac mass at $p$.
The equation \eqref{eq:gcsh} arises from \cite{BCT} 
which  generalizes some physical aspects of the Abelian Chern-Simons-Higgs model \cite{HKP,JP}.
By the maximum principle, if $U$ is a solution of \eqref{eq:gcsh}, then
\begin{equation}
\label{eq:U<0}
U\le 0 \qon \Omega \setminus \Gamma .
\end{equation}
In fact, if $x_0$ is a positive maximum point of $U$, then $x_0 \ne p_j$ and
\[ 0< (1-e^{U(x_0)} ) \Delta U (x_0)  +  \frac{1}{\eps^2} e^{U(x_0)} (e^{U(x_0)}-1)^2 = e^{U(x_0)} |\nabla U(x_0)|^2 =0,
\]
which is a contradiction.
If we set
 \begin{align}
 \label{eq:rho}
   u=\varrho(U)=1+U-e^U,
 \end{align}
 then $\varrho:(-\infty, 0] \to (-\infty,0]$ is  increasing and thus has the inverse $U=\psi(u)$.
 So, we define a function $\psi:\rone \to \rone$ by
 \begin{equation}
 \label{eq:psi def}
 \psi(u) =\left\{ 
 \begin{aligned} 
 & \varrho^{-1} (u) && \text{for} \quad u\in (-\infty , 0],\\
 &0 && \text{for} \quad  u \in (0,\infty) .
 \end{aligned} \right.
 \end{equation} 
Then, we can rewrite \eqref{eq:gcsh} as
\begin{align}
\label{eq:gcsh u}
\Delta u = - \frac{1}{\eps^2} e^{\psi(u)} \big( e^{\psi(u)} -1 \big)^2 + 4\pi \sum^{d}_{j=1}n_j \delta_{p_j}   \qin \Omega.
\end{align}

Let us define an auxiliary function
\begin{align*}
  u_0 (x) = - 4\pi \sum^{d}_{j=1} n_j G(x, p_j),
\end{align*} 
where  $G$  is the Green function such that
 \[ - \Delta_x G(x,y) = \delta_y (x) - \f{1}{|\Omega|} \qand  \int_\Omega G(x,y) dy =0. \]
 Here $|\Omega|$ denotes the area of $\Omega$.
Then $u_0$ satisfies
 \[ \Delta u_0 = 4\pi \sum^{d}_{j=1}n_j \delta_{p_j} - \f{4\pi N}{|\Omega|}  \qand  \int_\Omega u_0 =0. \]
If we set $v:=u-u_0$, then
 \begin{align}
 \label{eq:v}
    \Delta v = -\f{1}{\varepsilon^2} e^{\psi(u_0+v)}(e^{\psi(u_0+v)}-1)^2 + \f{4\pi N}{|\Omega|}  \qin \Omega.
 \end{align}
Since $\psi(t)=0$ for $t \ge 0$ and $e^t(e^t-1)^2 \le 4/27$ for $t \le 0$,  we get
\begin{align*}
    4\pi N =  \f{1}{\varepsilon^2}  \int_\Omega e^{\psi(u_0+v)}(e^{\psi(u_0+v)}-1)^2
               \leq  \f{4}{27 \varepsilon^2} |\Omega| .
\end{align*} 
Thus, we obtain that a necessary condition for the existence of \eqref{eq:v} is 
 \begin{equation}
 \label{eq:eps *}
  0< \eps < \eps_*:= \sqrt{\f{|\Omega|}{27\pi N}}. 
  \end{equation}
 In view of \eqref{eq:U<0}, we seek for a solution $v$ of  \eqref{eq:v} so that   $u_0+v \le 0$ on
$ \Omega\setminus \Gamma$.
Regarding the equation \eqref{eq:v}, the following theorem is known from \cite{xh}.


\begin{thma}

\begin{itemize}
\item[{\rm (i)}]
There exists $\eps_c \in (0, \eps_*)$ such that \eqref{eq:v} admits  a maximal solution $v_{1, \eps}$  to
 \eqref{eq:v} for $0<\eps < \eps_c$.
  Moreover, $v_{1, \eps}< -u_0$. 

\item[{\rm (ii)}]
For $0< \eps_1 < \eps_2 < \eps_c$, it holds that $v_{1, \eps_2}<v_{1, \eps_1}$,  and the limit
 \[ \phi := \inf_{ \eps \nearrow \eps_c} v_{1, \eps}  \]
 is a solution of \eqref{eq:v} for $\vae=\vae_c$.
\end{itemize}
\end{thma}

In Theorem A, the maximality means that if $v$ is any solution of \eqref{eq:v}, then $v<v_{1, \eps}$.
Based on the results above, we construct another solution of \eqref{eq:v}.
The main result of this paper is the following theorem.

\begin{theorem}
\label{thm:main}
For each $0<\eps < \eps_c$, there exists a second solution of  \eqref{eq:v}.
\end{theorem}

\begin{remark}
We explain the main contribution of this paper.
Consider the classical Abelian Chern-Simons-Higgs equation 
\begin{align}
\label{eq:csh-orig-V}
 \Delta V    =    \frac{1}{\eps^2} e^{u_0+V}(e^{u_0+V}-1) +\f{4\pi N}{|\Omega|} 
\end{align}
for anyonic planar physics  \cite{HKP,JP}. 
There are two important results for the existence of solutions of \eqref{eq:csh-orig-V}.
First, Caffarelli and Yang proved in \cite{cy} the existence of maximal solution $V_{1, \vae}$ 
of \eqref{eq:csh-orig-V} for $0<\vae<\hat{\vae}_c < \hat{\vae}_*:= \sqrt{ |\Omega|/(16\pi N)}$.
Second, Tarantello obtained in \cite{ta} a second solution by using the mountain pass structure 
of the  functional associated with \eqref{eq:csh-orig-V}.
See also \cite{CI00,CFL,Ch05JMP,CK08,LY13,XH13,ta07,ta-book} for other results.

For our equation \eqref{eq:v},  X. Han  proved  Theorem A as in \cite{cy} 
by the super- and sub-solution method.
Then, one may ask whether  a second solution exists as in  \cite{ta}.
In this paper, we obtain it via the Leray-Schauder degree theory   based on the method in \cite{HaLi,HS2}.
We emphasize that our idea is different from the method used in   \cite{ta}.
We believe that the method in this paper can be applicable to other types of self-dual Chern-Simons-Higgs equations.
\qed
\end{remark}


\section{Proof of  Theorem \ref{thm:main}}

 Let us fix $\vae_0 \in (0, \eps_c)$.
 We will show that $\eqref{eq:v}_{\vae=\vae_0}$ has a solution   which is different from $v_{1, \vae_0}$.
 We will employ the Leray-Schauder degree theory as in \cite{HaLi,HS2}.
 We choose $\eps_1$ so that $0< \vae_0 < \eps_c < \eps_* < \vae_1$.
  Then,  $\eqref{eq:v}_{\vae=\vae_1}$ has no solution.

 Define a space $\calZ = C(\Omega)$
 and take a constant $\leps > 4/\vae^2$.
   Let us denote
\begin{align*}
g(\eps, v)&:= - \f{1}{\varepsilon^2} e^{\psiuv}(e^{\psiuv}-1)^2,\\
f(\eps, v)&:= g(\eps, v) - \leps v + \f{4\pi N}{|\Omega|}.
\end{align*}
We denote $g_v= \pa g /\pa v$ and $f_v = \pa f / \pa v$.
Using the relation $\psi = \varrho^{-1}$ in \eqref{eq:rho}, we obtain
  \begin{align*}
     g_v (\eps, v) & = -\f{1}{\eps^2} \psi'(u_0+v) e^{\psiuv} (e^{\psiuv}-1)(3e^{\psiuv}-1) \\
                          & =  \f{1}{\eps^2} e^{\psiuv}(3e^{\psiuv}-1).
  \end{align*}
  Thus, when $u_0+v\le 0$, 
  \begin{equation}
  \label{eq:la-vae ineq}
  | g_v (\eps, v) | \le \frac{4}{\vae^2} <\la_\vae.
  \end{equation}

Now, \eqref{eq:gcsh u} is rewritten as
\[  \Delta v - \leps v =f(\eps, v)  .
\]
For each $v \in \calZ$, there exists a unique $w \in C^2(\Omega)$ such that
  $\Delta w - \leps w = f(\eps, v)$.
So, we can define   an operator $T_\eps : \calZ \to \calZ$ by $ T_\eps (v) = (\Delta - \leps)^{-1} f(\eps, v)$.
Since $C^2(\Omega) \Subset   C(\Omega)$, it is clear that this operator is compact.
 Let us also introduce a homotopy $H: [\eps_0, \eps_1] \times \calZ \to \calZ$ defined by $H(\eps, v)=v - T_\eps (v)$. 

 \begin{lemma}
 \label{lem:H}
   If $H(\eps, v)=0$ for $(\eps, v) \in [\vae_0, \vae_1]\times \calZ$, then $\| v \|_\calZ  < R_0$ for some $R_0>0$.
 \end{lemma}
 \begin{proof}
    Suppose that  there exists a sequence $(\eps_n, v_n) \in  [\eps_0, \eps_1]\times \calZ$ such that
      \begin{align}
      \label{eq:vn}
       H(\eps_n, v_n)=0 \qand  \| v_n \|_\calZ \to \infty.
     \end{align}
   Let us decompose $v_n:= \tilde v_n + \mu_n$   with
    $\int_\Omega \tilde v_n =0$ and $\mu_n = |\Omega|^{-1} \int_\Omega v_n  $.
    Then,
      \begin{align}
      \label{eq:tilde_vn}
      \Delta \tilde v_n = - \f{1}{\eps_n^2} e^{\psi(u_0+v_n)}(e^{\psi(u_0+v_n)}-1)^2 + \f{4\pi N}{|\Omega|} .
      \end{align}
      By Theorem A, $u_0+v_n < u_0+v_{1, \vae_n}<0$ which implies 
     $\Delta \tilde v_n  \in  L^\infty(\Omega)$ uniformly. 
  So, we have $\tilde v_n \in W^{2,p}(\Omega)$ for $p>1$, which implies $\| \tilde v_n \|_\calZ \leq C$ for some $C>0$.
Integrating the inequality $u_0 + \tilde v_n + \mu_n \le 0$ , we see that $\mu_n \le 0$.
Hence, we conclude that $\mu_n \to -\infty$.
 It comes from \eqref{eq:tilde_vn} that
 \[ 0=   4\pi N -  \f{1}{\eps_n^2} \int_\Omega   e^{\psi(u_0+\tilv_n+\mu_n)}(e^{\psi(u_0+\tilv_n+\mu_n)}-1)^2 \to 4\pi N \]
 since $\psi (t)\to -\infty$ as $t \to -\infty$.
  Thus we complete the proof.
 \end{proof}

Let
$B_{R} =\{ v \in \calZ : \; \| v\|_\calZ < R \}$ for $R \ge R_0$ and
    $$\calY= \{ v \in \calZ : \; \phi < v < -u_0 \}$$
where $\phi$ is defined in Theorem A.

 \begin{lemma}
 \label{lem:B_R}
  $H(\eps, v)\ne 0$ for $(\eps, v) \in [\vae_0, \vae_1]\times \pa (B_R \cap \calY)$  for all $R\ge R_0$.
 \end{lemma}
 \begin{proof}
  Suppose that $H(\vae_0, v)=0$ for some    $v \in \partial(B_R \cap Y) \subset \partial B_R \cup \partial Y$.
  By Lemma \ref{lem:H}, it suffices to consider     the case of $v \in \partial Y$.
  Then, one of the following is true:
  \begin{itemize}
  \item[{(i)}] there exists  $ x_0 \in \Omega \setminus \Gamma$ such that 
  $ v(x_0)=-u_0(x)$ and $v \le  -u_0$ on $\Omega \setminus \Gamma$;
  \item[{(ii)}] there exists $x_1 \in \Omega$ such that $v(x_1)=\phi(x_1)$ and $v \geq \phi$ on $\Omega$.
  \end{itemize}  
  Since $H(\vae_0, v)=0$,  $v$ satisfies $\eqref{eq:v}_{\vae=\vae_0}$.
   So, $v<v_{1, \vae_0}<-u_0$ by Theorem A which implies that (i) is excluded. 
 Suppose that (ii) is true.
  Since $\vae_0 < \vae_c$, we have
  \begin{equation}
  \label{eq:phi ineq}
   \begin{aligned}
      \Delta \phi & = - \f{1}{\varepsilon_c^2} e^{\psi(u_0+\phi)}(e^{\psi(u_0+\phi)}-1)^2 + \f{4\pi N}{|\Omega|} \\
                       & >   - \f{1}{\varepsilon_0^2} e^{\psi(u_0+\phi)}(e^{\psi(u_0+\phi)}-1)^2+ \f{4\pi N}{|\Omega|}
                           = g(\eps_0, \phi)  + \f{4\pi N}{|\Omega|}.
 \end{aligned}
 \end{equation}
  Thus, by the   maximum principle, we are led to a contradiction:
  \begin{align*}
0 & \leq  \Delta  (v-\phi) (x_1) <  g \big(\eps_0, v (x_1) \big) - g\big(\eps_0, \phi (x_1) \big)  =0.
\qedhere
  \end{align*}
 \end{proof}

By Lemma \ref{lem:B_R},   the Leray-Schauder degree 
$ \deg(H(\eps_0, \cdot), B_R \cap \calY, 0)$ is well-defined for $R > R_0$.
We want to show that  $\deg (H(\eps_0, \cdot),  B_R \cap \calY, 0)=1$.
For this purpose, we will define a new homotopy as follows.
 Given $(t,v) \in [0,1]\times \calZ$, let us consider the following equation
  \[ (\Delta - \lambda_{\eps_0})w = t f(\eps_0, v) + (1-t) f(\eps_0, v_{1, \eps_0}) \qon \Omega. \]
Then we can define  a new compact operator $K_t: \calZ \to \calZ$ by
 \[ w=K_t (v):= (\Delta - \lambda_{\eps_0})^{-1} \Big(t f(\eps_0, v) + (1-t) f(\eps_0, v_{1, \eps_0})\Big).\]
We also define a new homotopy
$J: [0,1] \times \calZ \to \calZ$ by $J(t,v)=v-K_t (v)$.

\begin{lemma}
\label{lem:J not zero}
$J(t,v) \ne 0$ for  $(t,v) \in [0,1] \times   \partial \calY$.
\end{lemma}
\begin{proof}
Suppose that $J(t_0, v_0)=0$ for some   $t_0 \in [0,1]$  and $v_0 \in \partial Y$.
Then
\begin{align}
 \label{eq:v0}
 (\Delta - \lambda_{\eps_0})v_0 = t_0 f(\eps_0, v_0)+ (1-t_0) f(\eps_0, v_{1, \eps_0}).
\end{align}
Since $v_0 \in \pa \calY$, either
  \[ \text{(i) $\phi(x_1) = v_0 (x_1)$ for some $x_1$ and $v_0 \ge  \phi$ on $\Omega  $,}\]
   or
   \[ \text{(ii) $-u_0(x_2) = v_0 (x_2)$ for some $x_2 \ne p_j$ and $v_0 \leq -u_0$ on $\Omega\setminus \Gamma  $.}
   \]
   First, assume the case of (i).
By \eqref{eq:phi ineq},
 \begin{align}
 \label{eq:phi_c}
  (\Delta - \lambda_{\eps_0}) \phi >  f(\eps_0, \phi) = t_0  f(\eps_0, \phi) + (1-t_0)  f(\eps_0, \phi).
 \end{align}
 Using \eqref{eq:v0} and \eqref{eq:phi_c} at $x_1$ and the inequality \eqref{eq:la-vae ineq}, 
 we obtain by the Mean Value Theorem that 
 \begin{align*}
   0  & \leq (\Delta - \lambda_{\eps_0}) (v_0-\phi)(x_1) \\
       & < t_0 \Big[ g(\eps_0, v_0) - g(\eps_0, \phi) - \lambda_{\eps_0} (v_0-\phi) \Big](x_1) \\
       &     \quad    +  (1-t_0) \Big[ g(\eps_0, v_{1, \eps_0}) - g(\eps_0, \phi) - \lambda_{\eps_0} ( v_{1, \eps_0}-\phi) \Big](x_1) \\
       & = t_0 \Big[ g_v\big(\eps_0, \xi_1(x_1)\big) - \lambda_{\eps_0} \Big] \big(v_0 (x_1)- \phi(x_1)\big)\\
      & \quad       + (1-t_0)  \Big[ g_v(\eps_0, \xi_2(x_1)) - \lambda_{\eps_0} \Big] (v_{1, \eps_0}(x_1) - \phi (x_1))    \\
       & =    (1-t_0)  \big( g_v(\eps_0, \xi_2(x_1)) - \lambda_{\eps_0} \big) \big(v_{1, \eps_0}(x_1) - \phi (x_1)\big) < 0,
 \end{align*}
which is a contradiction.
Here $\phi < \xi_1 < v_0$ and $\phi < \xi_2 < v_{1, \eps_0}$.

Second, we consider the case of (ii).
Similarly to the first case, we see that
 \begin{align*}
 0  & \geq (\Delta - \lambda_{\eps_0}) (v_0 + u_0)(x_2) \\
     & = t_0 \Big[ g(\eps_0, v_0) - g(\eps_0, -u_0) - \lambda_{\eps_0} (v_0 + u_0) \Big] (x_2)\\
     &  \quad      +  (1-t_0) \Big[ g(\eps_0,  v_{1, \eps_0}) - g(\eps_0, -u_0) - \lambda_{\eps_0} ( v_{1, \eps_0} + u_0) \Big] (x_2)\\
     & =  (1-t_0)  \Big[ g_v\big(\eps_0, \eta(x_2)\big) - \lambda_{\eps_0} \Big] (v_{1, \eps_0}(x_2) + u_0(x_2))  > 0
 \end{align*}
  for some $v_0 <\eta <-u_0$.
 This   completes the proof.
\end{proof}

\begin{lemma}
\label{lem:deg}
 There exists $R_1>0$ so that $\| v \|_{\calZ} \leq R_1$ provided $J(t,v) = 0$ for $(t,v) \in [0,1] \times  \calY$.
\end{lemma}
\begin{proof}
  From  $J(t,v) = 0$, it follows that
  \begin{equation}
  \label{eq:v2}
  \begin{aligned}
  \Delta v - \lambda_{\eps_0} v& = t f(\eps_0, v)+(1-t) f(\eps_0, v_{1, \eps_0})\\
  &=t g(\eps_0, v) - t\la_{\eps_0} v + \f{4t\pi N}{|\Omega|}+(1-t) f(\eps_0, v_{1, \eps_0}).
  \end{aligned}
  \end{equation}
    Let $v=\tilde v + \mu$ where
    $\int_\Omega \tilde v =0$ and $\mu = |\Omega|^{-1}\int_\Omega v $.
  By the property of $\psi$, it holds that $f\big (\vae_0, v(x) \big) $ is uniformly bounded on $\Omega$.
  Thus, by multiplying \eqref{eq:v2} by $-\tilde v$, we obtain 
   \begin{align*}
     \int_\Omega |\nabla \tilde v|^2 + \lambda_{\eps_0} \int_\Omega |\tilde v|^2
     &  =   -\int_\Omega t \ f(\eps_0, v) \tilde v  - \int_\Omega (1-t) f(\eps_0, v_{1, \eps_0}) \tilde v \\
     & \leq C \int_\Omega |\tilde v| \leq C \Big(\int_\Omega |\nabla \tilde v|^2 \Big)^{1/2}
   \end{align*}
  by the Poincar\'{e} inequality. Then $\int_\Omega |\nabla \tilde v|^2 \leq C$.
  Furthermore, integration of \eqref{eq:v2} yields that
   \[ -\lambda_{\eps_0} \mu |\Omega| =  4\pi N t - \lambda_{\eps_0} \mu |\Omega| t + t \int_\Omega g(\eps_0, v)
    + (1-t)  \int_\Omega f(\eps_0, v_{1, \eps_0}),    \]
  which gives
   \[ \big| \lambda_{\eps_0} \mu (1-t)  \big| \leq C. \]
   Then we see that $\Delta \tilde v \in L^2(\Omega)$ from \eqref{eq:v2}.
 So, by the elliptic regularity, we have $\| \tilde v \|_\infty < C$.
   Taking a nonsingular point $x_* \ne p_j$, we see  the finiteness of $\mu$ since
    \[ |\mu| \leq |v(x_*)| + |\tilde v(x_*)| \leq  |\tilde v(x_*)| + |u_0(x_*)| + |\phi (x_*)| \leq C.
    \qedhere \]
\end{proof}

\noindent
{\bf Proof of  Theorem \ref{thm:main}:} \;
Let us choose $R \geq \max\{ R_0, R_1\}$.
In accordance with Lemma \ref{lem:J not zero} and Lemma \ref{lem:deg}, $J(t,v) \ne 0$ 
on $\partial (B_R \cap \calY) \subset (\partial  B_R \cap \calY) \cup  (  B_R \cap \partial \calY) $  for all $t \in [0,1]$. 
 Thus the degree of $J$ is well-defined on $B_R \cap \calY$.
Since $v_{1, \eps_0}$ is the unique solution of $J(0, v)=0$,
we conclude that
    \[ \deg(H(\eps_0, \cdot), B_R \cap \calY, 0) =  \deg(J(1, \cdot), B_R \cap \calY, 0) = \deg(J(0, \cdot), B_R \cap \calY, 0)=1 \]
 by the homotopy invariance property of degree.
On the other hand, we see that $ \deg(H(\eps_1, \cdot), B_R, 0) =0$ 
since there is no solution of \eqref{eq:v} by \eqref{eq:eps *}.
Using these identities, we deduce that
  \begin{align*}
     0 & = \deg(H(\eps_1, \cdot), B_R, 0) = \deg (H(\eps_0, \cdot), B_R, 0)  \\
        & =  \deg(H(\eps_0, \cdot), B_R \cap \calY, 0) +  \deg(H(\eps_0, \cdot), B_R \setminus \calY, 0)\\
        &= 1+   \deg(H(\eps_0, \cdot), B_R \setminus \calY, 0) .
  \end{align*}
 Thus  $\deg(H(\eps_0, \cdot), B_R \setminus \calY, 0)=-1$, which implies that
 there exists a solution $v_{2, \vae_0}$ on $B_R \setminus \calY$ for $H(\eps_0, v)=v-T_{\eps_0}(v)=0$.
 Hence we conclude that $v_{2, \vae_0}$ is a solution of $\eqref{eq:v}_{\vae=\vae_0}$.
 Since $v_{1, \vae_0} \in \calY$, it is obvious that $v_{1, \vae_0} \ne v_{2, \vae_0}$.
 This finishes the proof of  Theorem \ref{thm:main}.

%


\bigskip


\begin{thebibliography}{1}

\bibitem{BCT}
J. Burzlaff, A. Chakrabarti and D.H. Tchrakian,
Generalized self-dual Chern-Simons vortices,
Phys. Lett. B 293 (1992) 127-131.

\bibitem{CI00} D. Chae and O. Y. Imanuvilov,
         The existence of non-topological multivortex solutions in the relativistic
         self-dual Chern-Simons theory,
         Comm. Math. Phys. 215 (2000),  119-142.

\bibitem{cy}
L. Caffarelli, Y. Yang,
Vortex condensation in the Chern-Simons Higgs model: An existence theorem,
Comm. Math. Phys  168 (1995), 321-336.

\bibitem{CFL}
H. Chan, C.-C. Fu, and C.-S. Lin,
Non-topological multi-vortex solutions to  the self-dual Chern-Simons-Higgs equation,
Comm. Math. Phys.  231 (2002), 189-221.

\bibitem{Ch05JMP}
K. Choe,
Uniqueness of the topological multivortex solution in the self-dual Chern-Simons theory
J. Math. Phys. 46 (2005), 012035.

\bibitem{CK08}
K. Choe and N. Kim,
Blow-up solutions of the self-dual Chern-Simons-Higgs vortex equation,
Ann. Inst. H.P. Anal. Non Lin\'{e}aire 25 (2008), 318-338.



\bibitem{HaLi}
J. Han and C.-S. Lin,
Multiplicity for self-dual condensate solutions in the Maxwell-Chern-Simons O(3) sigma model, 
Comm  PDE 39  (2014), 1424-1450.

\bibitem{HS2}
J. Han and J. Sohn,
On the self-dual Einstein-Maxwell-Higgs equation on compact surfaces,
Disc. Cont. Dyn. Syst.    39 (2019), 819-839.

\bibitem{HKP}
J. Hong, Y. Kim, and P. Y. Pac,
Multivortex Solutions of  the Abelian Chern-Simons-Higgs Theory,
Phys. Rev. Lett.  64  (1990), 2230-2233.

\bibitem{JP}
R. Jackiw and S.-Y. Pi,
Soliton solutions to the    gauged nonlinear Schr\"{o}dinger equations on the plane,
Phys. Rev. Lett.  64 (1990),    2969-2972.

\bibitem{LY13}
C.-S. Lin and S. Yan,
Existence of Bubbling Solutions for Chern-Simons Model on a Torus,
Arch. Rat. Meach. Anal. (2013), 353-392.

\bibitem{XH13}
X. Han,
The existence of multi-vortices for a generalized self-dual Chern–Simons model,
Nonlinearity 26 (2013), 805-835.

\bibitem{xh}
X. Han,
Existence of doubly periodic vortices in a generalized Chern-Simons model,
Nonlin. Anal. Real World Appl. 16 (2014), 90-102.

\bibitem{ta}
 G. Tarantello,
 Multiple condensate solutions for the Chern-Simons-Higgs theory,
 J. Math. Phys 37 (8), 1996, 3769-3796.

 \bibitem{ta07}
 G. Tarantello,
 Uniqueness of selfdual periodic Chern-Simons vortices of topological-type,
Calv. Var. 29 (2007), 191-217.

\bibitem{ta-book} G. Tarantello,
         Selfdual Gauge Field Vortices,
         Progress in Nonlinear Differential   Equations and their applications Vol 72,
         Birkhauser, 2008.


\end{thebibliography}
\end{document}